\documentclass[12pt]{amsart}

\usepackage[english]{babel}
\usepackage{vmargin,graphicx,amsmath,amssymb,color,subfigure,enumerate,csquotes, dsfont, mathrsfs}
\numberwithin{equation}{section}

\usepackage{color}
\definecolor{gr}{rgb}   {0.,   0.69,   0.23 }
\definecolor{bl}{rgb}   {0.,   0.5,   1. }
\definecolor{mg}{rgb}   {0.85,  0.,    0.85}
\definecolor{or}{rgb}   {0.9,  0.5,   0.}

\definecolor{webred}{rgb}{0.75,0,0}
\definecolor{webgreen}{rgb}{0,0.75,0}
\usepackage[citecolor=webgreen,colorlinks=true,linkcolor=webred]{hyperref}

\newtheorem{theorem}{Theorem}[section]
\newtheorem{lemma}[theorem]{Lemma}
\newtheorem{notation}[theorem]{Notation}

\newtheorem{proposition}[theorem]{Proposition}
\newtheorem{corollary}[theorem]{Corollary}

\newtheorem{definition}[theorem]{Definition}
\newtheorem{remark}[theorem]{Remark}

\newcommand{\N}{\mathbb{N}}
\newcommand{\R}{\mathbb{R}}

\newcommand{\diag}{\mathsf{diag}}

\newcommand{\Hil}{\mathcal{H}}
\newcommand{\HS}{\mathcal{L}_2(\Hil)}
\newcommand{\SH}{\mathcal{S}(\Hil)}
\newcommand{\AH}{\mathcal{A}(\Hil)}
\newcommand{\D}{\mathcal{D}(\Hil)}
\newcommand{\UH}{\mathcal{U}(\Hil)}

\newcommand{\dx}{\,\mathrm{d}}

\begin{document}
\title[Some remarks about flows of Hilbert-Schmidt operators]{Some remarks about flows of\\ Hilbert-Schmidt operators}
\author{B. Boutin}
\author{N. Raymond}

\begin{abstract}
This paper deals with bracket flows of Hilbert-Schmidt operators. We establish elementary convergence results for such flows and discuss some of their consequences.
\end{abstract}

\address{IRMAR, Universit\'e de Rennes 1, Campus de Beaulieu, F-35042 Rennes cedex, France}

\email{benjamin.boutin@univ-rennes1.fr}
\email{nicolas.raymond@univ-rennes1.fr}

\maketitle


\section{Introduction and results}

\subsection{Context}

Consider a real separable Hilbert space $\Hil$ equipped with the Hilbertian inner product $\langle\cdot,\cdot\rangle$ and the associated norm $\|\cdot\|$. The set of bounded operators from $\Hil$ to $\Hil$ will be denoted by $\mathcal{L}(\Hil)$ and the associated norm by $\|\cdot\|$. Let $(e_n)_{n\in\N}$ be a Hilbertian basis of $\Hil$ and denote by $\D$ the subset of bounded diagonal operators with respect to that basis:
\[
T\in\D \Leftrightarrow \forall n\in\N,\ \exists \lambda_n\in\R,\ Te_n = \lambda_n e_n.
\]
The purpose of this paper is to analyze the structural and convergence properties of some flows of Hilbert-Schmidt operators for which $\D$ turns out to be an attractive Hilbert-Schmidt subspace. We obtain an alternative convergence proof that applies for example to the flows (in infinite dimension) of Brockett, Toda and Wegner. The present paper is mainly motivated by the work by Bach and Bru \cite{BB10} where they tackle in particular the example of the Brockett flow in infinite dimension. We extend one of their results to a wider class of flows of Hilbert-Schmidt operators. We may notice here that such flows of operators have deep applications in mathematical physics. For instance, as recently proved by Bach and Bru in \cite{BB15}, they can be used to diagonalize unbounded operators in boson quantum field theory. 

This paper also aims at being a modest review of the methods and insights appearing in finite dimension with the idea to extend them in infinite dimension. These methods are explained, for instance, in the seminal papers \cite{CN88} and \cite{Bro89} (see also \cite{Bro91}, and \cite{W94} for the Wegner flow), but also in the nice survey \cite{Chu08}. Their relations with numerical algorithms such as the QR or LU decompositions are very deep and the idea of interpreting discrete algorithms as samplings of continuous flows has now a long story (see, for instance, \cite{R54a, R54b} and \cite{WL90}). Commenting this idea would lead us far beyond the scope of this paper, but we will briefly discuss it when describing our numerical illustrations.

\subsection{Results}
Let us now describe the results of this paper. For that purpose, we introduce some general notations.
\begin{notation}
Below are the notations that we will constantly use in the paper.
\begin{enumerate}[\rm (a)]
\item We denote $\HS$ the set of the Hilbert-Schmidt operators on $\Hil$ and the associated norm is $\|\cdot\|_{\mathsf{HS}}$. In other words, $H\in\HS$ means that, for any Hilbertian basis $(e_{n})_{n\in\mathbb{N}}$, we have
\[\sum_{n\geq 0}\|He_{n}\|^2<+\infty\,,\]
and this last quantity is independent from the Hilbert basis and is denoted by $\|H\|^2_{\mathsf{HS}}$ and we have $\|H\|\leq\|H\|_{\mathsf{HS}}$. We also recall that the Hilbert-Schmidt operators are compact and that, if $(\lambda_{n}(H))_{n\in\mathbb{N}}$ is the non-increasing sequence of the eigenvalues of $H$, we have
\[\|H\|_{\mathsf{HS}}^2=\sum_{n=0}^{+\infty}|\lambda_{n}(H)|^2\,.\]
\item We will denote by $\SH$, $\AH$ and $\mathcal{U}(\Hil)$ the sets of the bounded symmetric, bounded skew-symmetric and unitary operators respectively. We also introduce $\mathcal{S}_{2}(\Hil)=\mathcal{S}(\Hil)\cap \mathcal{L}_{2}(\Hil)$ and $\mathcal{A}_{2}(\Hil)=\mathcal{A}(\Hil)\cap \mathcal{L}_{2}(\Hil)$. They are Hilbert spaces equipped with the scalar product $(A,B)\mapsto\mathsf{Tr}(A^\star B)$.

\item In the sequel we will use a fixed Hilbert basis $(e_n)_{n\in\N}$ and for any $H\in\HS$ and any integers $i,j\in\N$, we will denote $h_{i,j}(t)=\langle H(t)e_i,e_j \rangle$.
\end{enumerate}

\end{notation}

\subsubsection{An elementary convergence result}
We will consider infinite systems determined by matrices that are balanced in the sense of the following definition.
\begin{definition}
A family of real-valued functions $(g_{i,j})_{i,j\in\N}$ defined on $\R$ is said to be \emph{balanced} if
\[
\forall i,j\in\N,\ \exists c_{i,j}\in L^\infty(\R,\R_+),\ \forall t\in\R,\ |g_{i,j}(t)| \leq c_{i,j}(t) |g_{j,i}(t)|\,.
\]
It is said to be \emph{pointwise bounded} if, for all $t\in\R$, $\displaystyle{\sup_{i,j} |g_{i,j}(t)|<+\infty}$.
\end{definition}

Let us now state one of the main results of this paper. It is a general convergence result for infinite systems of differential equations. As we will see, many different classical situations may be reduced to this result.
\begin{theorem}\label{gen-theo}
Consider $H\in\mathcal{C}^1(\R,\HS)$ such that
\begin{enumerate}[\rm i)]
\item the function $t\mapsto \|H(t)\|_{\mathsf{HS}}$ is bounded,
\item the family $(h_{i,j})_{i,j\in\N}$ is balanced,
\item there exists a pointwise bounded and balanced family of measurable functions $(g_{i,j})_{i,j\in\N}$ defined on $\R$ such that
\begin{equation}
\label{eq:ODEdiago}
\forall t\in\R,\ h_{i,i}'(t)=\sum_{j=0}^{+\infty}g_{i,j}(t)|h_{i,j}(t)|^2\,.
\end{equation}
\end{enumerate}
Suppose that there exists $T>0$ and a sign sequence $(\epsilon_\ell)_{\ell\in\N}\in\{-1,1\}^{\N}$, such that
\begin{equation}
\label{eq:assumpSign}
\forall t\geq T,\ \forall k,\ell\in\N,\ k\geq \ell,\ \epsilon_\ell g_{\ell,k}(t)\geq 0\,.
\end{equation}
Then we have the integrability property
\begin{equation}
\forall \ell\in\N,\ \sum_{j=0}^{+\infty}\int_T^{+\infty} |g_{\ell,j}(t)| |h_{\ell,j}(t)|^2\dx t<+\infty\,,
\end{equation}
and any diagonal term $h_{\ell,\ell}(t)$ converges at infinity to a real value denoted by $h_{\ell,\ell}(\infty)$.

Under the additional semi-uniform lower bound condition
\begin{equation}
\label{eq:assumpBound}
\forall \ell\in\N,\ \exists c_\ell>0,\ \forall j\neq \ell,\ \forall t\geq T,\ |g_{\ell,j}(t)|\geq c_\ell\,,
\end{equation}
we have
\begin{equation}\label{eq.integrability}
\sum_{j\neq \ell}\int_T^{+\infty} |h_{\ell,j}(t)|^2\,dt = \int_T^{+\infty} \|H(t)e_\ell - h_{\ell,\ell}(t)e_\ell\|^2\dx t <+\infty\,.
\end{equation}
Suppose moreover that, for any $\ell\in\N$, the function $t\mapsto \|H'(t)e_\ell\|$ is bounded, then for any $\ell\in\N$, $H(t)e_\ell$ converges to $h_{\ell,\ell}(\infty) e_\ell$.
\end{theorem}
\begin{remark}
Theorem~\ref{gen-theo} remains true if, instead of the assumption \eqref{eq:assumpSign}, we only assume that there exists a one-to-one integer mapping function $\phi:\N\to\N$, a time $T>0$ and a sign sequence $\epsilon_\ell\in\{-1,1\}$ ($\ell\in\N$) such that
\[
\forall t\geq T,\ \forall k\in\N,\ k\notin\{\phi(i),\ 0\leq i \leq \ell-1\},\ \epsilon_\ell g_{\phi(\ell),k}(t)\geq 0\,.
\]
\end{remark}
We will apply Theorem \ref{gen-theo} to \enquote{bracket} flows. Let us state a proposition that describes the fundamental property of such flows.
\begin{proposition}\label{prop.cor1}
Let $G:\SH\to\AH$ be a locally Lipschitzian function and let us consider the following Cauchy problem
\begin{equation}
\label{eq:ODE}
H' = [H,G(H)]\, ,\qquad t\geq 0\,, \qquad H(0)=H_{0}\in\mathcal{S}_{2}(\Hil)\,.
\end{equation}
There exists a unique global solution $H\in\mathcal{C}^1([0,+\infty),\mathcal{S}_{2}(\Hil))$ unitarily equivalent to the initial data $H_{0}$. More precisely there exists $U\in\mathcal{C}^1([0,+\infty),\UH)$ such that for any $t\geq 0$ one has $H(t) = U(t)^\star H_{0} U(t)$.
\end{proposition}
Let us now provide two corollaries of Theorem \ref{gen-theo}.

First let us introduce two convenient definitions.
\begin{definition}
Given a Hilbertian basis $(e_{n})_{n\in\mathbb{N}}$, we introduce 
\begin{enumerate}[i.]
\item the symmetric family of symmetric Hilbert-Schmidt operators $(E_{i,j})_{(i,j)\in\mathbb{N}^2}$ defined by
\[E_{i,j}(\cdot)=\begin{cases}
        \frac{1}{\sqrt{2}}(e^*_{j}(\cdot) e_{i}+e^*_{i}(\cdot) e_{j})& \text{if $ i\neq j$}\, \\
        e^*_{i}(\cdot) e_{i}  & \text{if $i=j$ }
               \end{cases}\,.
 \]

\item the skew-symmetric family of skew-symmetric Hilbert-Schmidt operators $(E^{\pm}_{i,j})_{(i,j)\in\mathbb{N}^2}$ defined by
\[E^\pm_{i,j}(\cdot)=\frac{1}{\sqrt{2}}(e^*_{j}(\cdot) e_{i}-e^*_{i}(\cdot) e_{j})\,.\]
\end{enumerate}
\end{definition}
Since the families $(E_{i,j})_{i\leq j}$ and $(E_{i,j}^\pm)_{i< j}$ are Hilbertian bases of $\mathcal{S}_{2}(\Hil)$ and $\mathcal{A}_{2}(\Hil)$ respectively, this motivates the following definition.
\begin{definition}
We say that $G\in\mathcal{L}\left(\mathcal{S}_{2}(\Hil),\mathcal{A}_{2}(\Hil)\right)$ is diagonalizable when there exists a Hilbertian basis $(e_{n})_{n\in\mathbb{N}}$ and a skew-symmetric family $(g_{i,j})$ such that
\[G(E_{i,j})=g_{i,j}E^\pm_{i,j}\,.\]
In this case, we say that $(g_{i,j})$ is the matrix-eigenvalue of $G$. Note that the matrix-eigenvalue of $G$ is balanced and pointwise bounded ($G$ is bounded).
\end{definition}

\begin{corollary}\label{cor.cor1'}
Let $G\in\mathcal{L}\left(\mathcal{S}_{2}(\Hil),\mathcal{A}_{2}(\Hil)\right)$ be a diagonalizable map such that its matrix-eigenvalue satisfies \eqref{eq:assumpSign} and \eqref{eq:assumpBound}. Then the Cauchy problem \eqref{eq:ODE} admits a unique global solution $H\in\mathcal{C}^1([0,+\infty),\mathcal{S}_{2}(\Hil))$ which weakly converges in $\mathcal{S}_{2}(\Hil)$ to a diagonal Hilbert-Schmidt operator $H_\infty$. Moreover, if $\alpha$ is a diagonal term of $H_{\infty}$ with multiplicity $m$, then $\alpha$ is an eigenvalue of $H_{0}$ of multiplicity at least $m$.
\end{corollary}
Then we can apply our theorem to the finite dimensional case and even get an explicit convergence rate. Note that this presentation does not involves a Lyapunov function and relies on the a priori convergence of $H(t)$. This is a way to avoid, as far as possible, the considerations related to the evolution of the corresponding unitary matrices $(U(t))_{t\geq 0}$ in $\mathcal{U}(\Hil)$ (that is not compact in infinite dimension).
\begin{corollary}\label{cor.dim.finie}
Consider a finite dimensional Hilbert space $\Hil$ of dimension $d$. Under the assumptions of Corollary \ref{cor.cor1'}, the solution of \eqref{eq:ODE} converges to a diagonal operator $H_{\infty}=\diag\left(\alpha_{\ell}\right)_{0\leq\ell\leq d-1}$ in $\mathcal{L}(\Hil)$ and the $\alpha_{\ell}$ are exactly the eigenvalues (with multiplicity) of $H_{0}$. If, in addition, the eigenvalues of $H_{0}$ are simple, then there exist $T, C>0$ such that, for all $t\geq T$, we have
\[\|H(t)-H_{\infty}\|\leq Ce^{-\gamma t}\,,\]
where $\displaystyle{\gamma=\inf_{\mathcal{T}^-}|g_{i,j}(\alpha_{i}-\alpha_{j})|>0}$ with 
\[\mathcal{T}^-=\{(i,j)\in\{0,\ldots, d-1\}^2 : i< j\quad\mbox{ and }\quad g_{i,j}(\alpha_{i}-\alpha_{j})<0\}\,.\]
\end{corollary}
Note that, in Corollary \ref{cor.dim.finie}, we have not to exclude special initial conditions to get an exponential convergence (see \cite[Theorem 3]{Bro91}).

The following proposition states that, even in infinite dimension, we may find exponentially fast eigenvalues of a generic Hilbert-Schmidt operator thanks to a bracket flow. Note that, in this case, we get a stronger convergence as the one established by Bach and Bru in \cite[Theorem 7]{BB10} for the Brockett flow and that we may also consider non symmetric operators.

\begin{proposition}\label{prop.House}
Assume that $H_{0}\in\mathcal{L}_{2}(\Hil)$ (not necessarily symmetric) is diagonalizable with eigenvalues $(\lambda_{j})_{j\geq 1}$ with decreasing real parts. 
Let us also assume that we may find $P\in\mathcal{L}(\Hil)$ such that $P H_{0} P^{-1}=\Lambda$ where $\Lambda=\diag(\lambda_{j})_{j\in\mathbb{N}}$ is a diagonal operator and such that the minors of $P$, $P_{J}=(\langle Pe_{i},e_{j}\rangle)_{0\leq i,j\leq J}$ are invertible for all $J\in\mathbb{N}$.

We use the choice 
\[G(H)=H^- -(H^-)^\star\,.\]
Then, we have, for all $\ell\in\mathbb{N}$,
\[H(t)e_{\ell}-\sum_{j=\ell+1}^{+\infty} \langle H(t)e_{\ell}, e_{j}\rangle e_{j}=\lambda_{\ell}e_{\ell}+\mathcal{O}\left(e^{-t\delta_{\ell}}\right)\,,\quad \delta_{\ell}=\min_{0\leq j\leq  \ell}\Re(\lambda_{j}-\lambda_{j+1})\,.\]
\end{proposition}
Note in particular that, in the symmetric case, $H(t)$ converges to a diagonal operator $H_{\infty}$ that has exactly the same eigenvalues as $H_{0}$ (even in infinite dimension). This property is slightly stronger than the last one stated in Corollary \ref{cor.cor1'}: for the Toda flow, in infinite dimension, no eigenvalue is lost in the limit.

\subsubsection{Organization of the paper}
Section \ref{sec.existence} is devoted to the proof of Theorem \ref{gen-theo}, Proposition \ref{prop.cor1} and Corollary \ref{cor.cor1'}. We also provide examples of applications of our results (to the Brockett, Toda and Wegner flows) in infinite dimension. In Section \ref{sec.finite.dim}, we discuss the case of finite dimension to the light of our results and discuss the exponential convergence of the flows (Corollary \ref{cor.dim.finie} and Proposition \ref{prop.House}).

\section{Existence and global convergence results}\label{sec.existence}

\subsection{Proof of the main theorem}\label{sec.proof.main}
This section is devoted to the proof of Theorem~\ref{gen-theo}. We start by stating two elementary lemmas.
\begin{lemma}\label{L1}
Consider $h\in\mathcal{C}^1(\R_+)$ a real-valued bounded function, $F\in L^1(\R_+)$ and $G$ a nonnegative measurable function defined on $\R_+$ such that
\[
\forall t\geq 0,\ h'(t) = F(t) + G(t).
\]
Then, the function $G$ is integrable and $h$ converges to a finite limit at infinity.
\end{lemma}

\begin{proof}
Observe that, due to the sign of $G$, one gets for all $x\geq 0$
\[
\int_0^x |G(t)|\dx t = \int_0^x G(t)\, dt = h(x)-h(0) - \int_0^x F(t)\dx t \leq 2 \sup_{\R_+}|h| + \int_0^{+\infty} |F(t)| \dx t \,.
\]
Therefore $G$ is integrable and, for all $x\geq 0$
\[
h(x) = h(0) + \int_0^{x} (F(t) + G(t))\dx t\, ,
\]
which converges as $x$ tend to infinity.
\end{proof}

\begin{lemma}\label{L1+Lip}
Let $w$ be a function defined on $\R_+$ with nonnegative values. Suppose that $w$ is Lipschitz continuous with Lipschitz constant $C>0$. Then, for any $x\geq 0$, one has
\[
w(x)^2 \leq 2C \int_x^{+\infty} w(y)\dx y \leq +\infty.
\]
Suppose moreover that $w$ is integrable, then $\displaystyle{\lim_{x\to +\infty}w(x) = 0}$.
\end{lemma}
\begin{proof}
Since $w$ is Lipschitzian, we have, for all $x, y\in\R$,
\[w(x)\leq C|x-y|+w(y)\,,\]
so that, integrating with respect to $y$ between $x$ and $x+\frac{w(x)}{C}$, we get
\[\frac{w(x)^2}{C}\leq \frac{w(x)^2}{2C}+\int_{x}^{x+\frac{w(x)}{C}}w(y)\dx y\,,\]
and the conclusion follows.
\end{proof}

We can now complete the proof of Theorem~\ref{gen-theo}.

\subsubsection{Convergence}
Consider first Equation \eqref{eq:ODEdiago} with $i=0$. Thanks to the sign property \eqref{eq:assumpSign} we have
\[
\epsilon_0 h'_{0,0}(t) = \sum_{j\in\N} |g_{0,j}(t)| |h_{0,j}(t)|^2\, .
\]
The boundedness of $t\mapsto\|H(t)\|_{\mathsf{HS}}$ implies the boundedness of $t\mapsto h_{0,0}(t)$ and thus Lemma~\ref{L1} implies the existence of a finite limit $h_{0,0}(\infty)$ and the integrability property
\[
\int_T^{+\infty} \sum_{j\in\N} |g_{0,j}(t)| |h_{0,j}(t)|^2\, dt < +\infty\, .
\]
Then, the rest of the proof is by induction on $\ell\in\N$. Let $\ell\geq 1$ be an integer and assume that, for any integers $0\leq k\leq \ell-1$, we have
\[
\int_T^{+\infty} \sum_{j\in\N} |g_{k,j}(t)| |h_{k,j}(t)|^2\, dt < +\infty\, .
\]
Equation \eqref{eq:ODEdiago} with $i=\ell$ and the sign assumption \eqref{eq:assumpSign} give
\[
\epsilon_\ell h'_{\ell,\ell}(t) = \sum_{k\leq \ell-1} \epsilon_\ell g_{\ell,k}(t) |h_{\ell,k}(t)|^2 + \sum_{j\geq \ell} |g_{\ell,j}(t)| |h_{\ell,j}(t)|^2\, .
\]
Again, $t\mapsto h_{\ell,\ell}(t)$ is bounded. Let us now prove the integrability of the first sum in the right hand side.
Since the families $(g_{i,j})$ and $(h_{i,j})$ are balanced, we get the following
\[
\exists C_\ell\geq 0,\ \forall t\geq T,\ \forall k\leq \ell-1,\ |g_{\ell,k}(t)||h_{\ell,k}(t)|^2 \leq C_\ell |g_{k,\ell}(t)||h_{k,\ell}(t)|^2.
\]
Therefore we have
\begin{equation}
\label{eq:partintegrability}
\begin{aligned}
\int_T^{+\infty} \sum_{k\leq \ell-1} |g_{\ell,k}(t)| |h_{\ell,k}(t)|^2\dx t & \leq C_\ell \sum_{k\leq \ell-1} \int_T^{+\infty} |g_{k,\ell}(t)||h_{k,\ell}(t)|^2\dx t\\
& \leq  C_\ell \sum_{k\leq \ell-1} \int_T^{+\infty} \sum_{j\in\N} |g_{k,j}(t)| |h_{k,j}(t)|^2\dx t < +\infty\, .
\end{aligned}
\end{equation}
Finally, Lemma~\ref{L1} applies and we get the existence of a finite limit $h_{\ell,\ell}(\infty)$ together with the integrability property:
\[
\int_T^{+\infty}  \sum_{j\geq \ell} |g_{\ell,j}(t)| |h_{\ell,j}(t)|^2\dx t < +\infty\, ,
\]
and then, thanks to \eqref{eq:partintegrability}, we find
\[
\int_T^{+\infty}  \sum_{j\in\N} |g_{\ell,j}(t)| |h_{\ell,j}(t)|^2\dx t < +\infty\, .
\]
This ends the proof of the integrability property and of the convergence of diagonal terms.
\subsubsection{Stronger convergence}
Consider now the additional assumption~\eqref{eq:assumpBound}. For any fixed $\ell\in\N$, we have
\[
\sum_{j\neq \ell}\int_T^{+\infty} |h_{\ell,j}(t)|^2\dx t \leq \dfrac{1}{c_\ell}\sum_{j\neq \ell}\int_T^{+\infty} |g_{\ell,j}(t)| |h_{\ell,j}(t)|^2\dx t<+\infty\,.
\]
For any $t\geq T$, we let 
\[w_{\ell}(t)=\sum_{j\neq \ell} |h_{\ell,j}(t)|^2\dx t =\|K_{\ell}(t)\|^2\qquad \mbox{ with }\qquad K_{\ell}(t)=H(t)e_\ell - h_{\ell,\ell}(t)e_\ell\,.\]
For $t\geq s\geq T$, the Cauchy-Schwarz inequality implies that
\[
|w_{\ell}(t)-w_{\ell}(s)|\leq \| K_{\ell}(t)-K_{\ell}(s) \| \|K_{\ell}(t)+K_{\ell}(s)\|\,.
\]
On one hand, we have
\[
\|K_{\ell}(t) + K_{\ell}(s)\| \leq 2\sup_{\sigma\geq T} \|H(\sigma)\|_{\mathsf{HS}}.
\]
On the other hand, we write
\[
\begin{aligned}
\| K_{\ell}(t)-K_{\ell}(s) \| &\leq \|H(t)e_\ell-H(s)e_\ell\| + |h_{\ell,\ell}(t)-h_{\ell,\ell}(s)| \\
& \leq \int_s^t \| H'(\sigma)e_\ell\|\, d\sigma + \int_s^t |h'_{\ell,\ell}(\sigma)|\dx\sigma\\
& \leq 2 \int_s^t \| H'(\sigma)e_\ell\|\, d\sigma \leq 2 M |t-s|\,,
\end{aligned}
\]
where $M$ is a uniform bound of $t\mapsto \|H'(t)e_\ell\|$. Finally, $w_{\ell}$ is a nonnegative and Lipschitz continuous function, which is moreover integrable thanks to \eqref{eq.integrability}. Therefore Lemma~\ref{L1+Lip} implies that $w_{\ell}(t)$ goes to $0$ as $t$ goes to $+\infty$.

\subsection{Proof of Proposition \ref{prop.cor1}}\label{sec.proof.cor1}

\subsubsection{Local existence}
Let us introduce the function $F:\mathcal{S}(\Hil)\to \mathcal{S}(\Hil)$ defined by
\begin{equation}\label{eq.F}
\forall H\in\mathcal{S}(\Hil)\,,\qquad F(H)=[H, G(H)]\,.
\end{equation}
Note indeed that, for all $H\in\mathcal{S}(\Hil)$, by using that $G$ is valued in $\mathcal{A}(\Hil)$, we have
\[F(H)^\star=(HG(H)-G(H)H)^{\star}=G(H)^\star H^\star-H^\star G(H)^\star=F(H)\,.\]
Moreover $F$ is locally Lipschitzian since $G$ is. By using the Cauchy-Lipschitz theorem in the Banach space $(\mathcal{S}(\Hil),\|\cdot\|)$ with the function $F$, we get the local existence of a solution of the Cauchy problem \eqref{eq:ODE}, on a maximal time interval $[0,T_{\max})$, with $T_{\max}>0$.

\subsubsection{Global existence}
Let us now show that $T_{\max}=+\infty$ and that, we have, for all $t\in[0,+\infty)$, $H(t)\in\HS$ and $\|H(t)\|_{\mathsf{HS}}=\|H_{0}\|_{\mathsf{HS}}$. Let us consider the following Cauchy problem:
\begin{equation}\label{eq.U}
U'=UG(H)\,,\qquad U(0)=\mathrm{Id}\,,
\end{equation}
where $H$ is the local solution of \eqref{eq:ODE} on $[0, T_{\max})$. Since the equation is linear, we know that there exists a unique solution $U$ defined on $[0, T_{\max})$. By using that $G(H(t))\in\mathcal{A}(\Hil)$, it is easy to check that, on $[0, T_{\max})$, we have $U^\star U=U U^\star=\mathrm{Id}$. Then an easy computation shows that the derivative of $t\mapsto U(t)H(t)U(t)^\star$ is zero and thus 
\[\forall t\in[0, T_{\max})\,,\qquad H(t)=U(t)^\star H_{0} U(t)\,.\]
From this, we infer that, for all $t\in[0, T_{\max})$, $H(t)\in\HS$ and $\|H(t)\|_{\mathsf{HS}}=\|H_{0}\|_{\mathsf{HS}}$. Therefore, on $[0, T_{\max})$, the function $H$ is bounded in $\mathcal{L}_{2}(\Hil)\subset\mathcal{L}(\Hil)$. Thus we get $T_{\max}=+\infty$.

\subsubsection{Linear case}
Let us now discuss the proof of Corollary \ref{cor.cor1'}. First we notice that, since $\mathcal{L}_{2}(\Hil)$ is an ideal of $\mathcal{L}(\Hil)$, the function $F$ defined in \eqref{eq.F} sends the Hilbert space $\mathcal{S}_{2}(\Hil)$ into itself. Thus, we may apply the Cauchy-Lipschitz theorem in this space and the global existence is ensured by the investigation in the previous section. Let us reformulate \eqref{eq:ODE} in the Hilbert basis $(e_{n})_{n\in\mathbb{N}}$. We easily have
\[h'_{i,i}(t)=2\langle G(H(t))e_{i}, H(t)e_{i}\rangle\,.\]
Then we notice that
\[H(t)=\sum_{k\leq \ell} h_{k,\ell}(t) E_{k,\ell}\,,\qquad G(H(t))=\sum_{k< \ell} g_{k,\ell}h_{k,\ell}(t)E^\pm_{k,\ell} \,.\]
We get 
\[H(t)e_{i}=\sum_{j=0}^{+\infty} h_{i,j}(t)e_{j}\,,\qquad  G(H(t))e_{i}=-\sum_{j=0}^\infty g_{i,j}h_{i, j}(t) e_{j}\,,\]
and thus
\[h'_{i,i}(t)=-2\sum_{j=0}^{+\infty} g_{i,j}|h_{i,j}(t)|^2\,.\]
Moreover we notice that
\[\|H'(t)\|\leq 2\|H(t)\|\|G(H(t))\|\leq 2\|G\|\|H(t)\|^2_{\mathsf{HS}}=2\|G\|\|H_{0}\|^2_{\mathsf{HS}}\,.\]

Therefore, thanks to the assumptions on $(g_{i,j})$, we may apply Theorem \ref{gen-theo}.

Then we let $\alpha_{\ell}=h_{\ell, \ell}(+\infty)$ and we notice that
\[\sum_{\ell=0}^{+\infty}\|H(t)e_{\ell}\|^2=\|H(t)\|^2_{\mathsf{HS}}=\|H_{0}\|^2_{\mathsf{HS}}\,,\]
and thus, with the Fatou lemma, we get
\[\sup_{\ell\in\mathbb{N}}|\alpha_{\ell}|^2\leq\sum_{\ell=0}^{+\infty}\alpha^2_{\ell}\leq\liminf_{t\to+\infty}\|H(t)\|^2_{\mathsf{HS}}=\|H_{0}\|^2_{\mathsf{HS}}\,.\]
We introduce the Hilbert-Schmidt operator $H_{\infty}:=\diag\left(\alpha_{\ell}\right)$ defined by
\[\forall x\in\Hil\,,\qquad H_{\infty}(x)=\sum_{\ell=0}^{+\infty} \alpha_{\ell}x_{\ell} e_{\ell}\,,\quad\mbox{ where }\quad x=\sum_{\ell=0}^{+\infty} x_{\ell} e_{\ell}\,.\]
In particular, we get that, for all $\ell\in\mathbb{N}$, $H(t)e_{\ell}$ converges to $H_{\infty}e_{\ell}$. Let us now consider $x\in\Hil$ and $\varepsilon>0$. There exists $L\in\mathbb{N}$ such that $\displaystyle{\|x-\sum_{\ell=0}^Lx_{\ell} e_{\ell}\|\leq\varepsilon}$. Since $\|H(t)\|\leq \|H(t)\|_{\mathsf{HS}}=\|H_{0}\|_{\mathsf{HS}}$ and $\|H_{\infty}\|\leq\|H_{0}\|_{\mathsf{HS}}$, we get that, for all $t\geq 0$,
\[\|(H(t)-H_{\infty})x\|\leq \|(H(t)-H_{\infty})\sum_{\ell=0}^Lx_{\ell} e_{\ell}\|+2\varepsilon\|H_{0}\|_{\mathsf{HS}}\,,\]
and, for $t$ large enough, it follows that
\[\|(H(t)-H_{\infty})x\|\leq 2\varepsilon\|H_{0}\|_{\mathsf{HS}}+\varepsilon\,.\]
Therefore $H(t)$ strongly converges to $H_{\infty}$ and thus $H(t)$ weakly converges to $H_{\infty}$ in $\mathcal{S}_{2}(\Hil)$.

Let us consider an eigenvalue $\alpha$ of $H_{\infty}$, with multiplicity $m$ and associated with the eigenvectors $e_{\ell_{1}},\ldots, e_{\ell_{m}}$. Then, for all $\varepsilon>0$, there exists $T>0$ such that, for all $j\in\{1,\ldots, m\}$,
\[\|H(T)e_{\ell_{j}}-\alpha e_{\ell_{j}}\|\leq\varepsilon\,.\]
or equivalently
\[\|H_{0}f_{j}-\alpha f_{j}\|\leq\varepsilon\,,\quad\mbox{ with }\quad f_{j}=U(T)e_{\ell_{j}}\,.\]
Therefore, the spectral theorem implies that
\[\forall\varepsilon>0\,,\quad \mathrm{range}\left(\mathds{1}_{[\alpha-\varepsilon,\alpha+\varepsilon]}(H_{0})\right)\geq m\,.\] 
This proves that $\alpha$ is an eigenvalue of $H_{0}$ with multiplicity at least $m$.

\subsection{Examples}
Let us now provide examples of applications of our convergence results.

\subsubsection{Brockett's choice} 
For a given $A=\diag(a_{\ell})\in\mathcal{D}(\Hil)$ with a non-increasing sequence $(a_{\ell})_{\ell\in\mathbb{N}}\in\R^\mathbb{N}_{+}$, we take $G_{1}(H) = [H,A]$. Since $\mathcal{L}_{2}(\Hil)$ is an ideal of $\mathcal{L}(\Hil)$, we have $G_{1} : \mathcal{S}_{2}(\Hil)\to\mathcal{A}_{2}(\Hil)$ and $\|G_{1}\|\leq 2\|A\|$. Moreover $G_{1}$ is diagonalizable and $g_{i,j}=a_{j}-a_{i}$.

\subsubsection{Toda's choice} 
If $H\in\mathcal{S}_{2}(\Hil)$, we introduce
\[H^-=\sum_{1\leq i\leq j} h_{i,j} e^*_{i} e_{j}\,\qquad \mbox{ and }\qquad G_{2}(H) = H^- - (H^-)^{\star}\,.\]
By definition, we have $G_{2} : \mathcal{S}_{2}(\Hil)\to\mathcal{A}_{2}(\Hil)$ and $\|G\|\leq 2\|H^-\|\leq 2\|H\|_{\mathsf{HS}}$. In addition, $G_{2}$ is diagonalizable with $g_{i,j}=-1$ for $i<j$.

\subsubsection{Wegner's choice} 
Let us finally discuss a non-linear example. We take $G_{3}(H) = [H,{\rm diag} (H)]$. Since $G_{3}$ sends the symmetric operators in the skew symmetric operators, we get that $G_{3}$ generates a flow that preserves the Hilbert-Schmidt norm. Thus the flow is global. Explicitly we have $H'=[H, [H,\diag H]]$. In particular, we get
\[h'_{i,i}(t)=2\sum_{j=0}^{+\infty} (h_{i,i}(t)-h_{j,j}(t))|h_{ij}(t)|^2\,.\]
If, for large times, the diagonal terms become distinct (in the sense of \eqref{eq:assumpSign} and \eqref{eq:assumpBound}), then, by Theorem \ref{gen-theo}, $H(t)$ converges to a diagonal matrix $H_{\infty}$. Actually, in this case, the convergence is exponential in finite dimension (see Section \ref{sec.conv.exp} where the linearization of the same kind of system is explained).

\section{Exponential convergence}\label{sec.finite.dim}

\subsection{A stronger convergence in finite dimension}
This section is devoted to the proof of Corollary \ref{cor.dim.finie}. 

\subsubsection{Reminder of the stable manifold theorem}
Let us recall a classical theorem.
\begin{theorem}\label{thm:SM}
Consider $F\in\mathcal{C}^1(\R^d,\R^d)$ such that $F(0) = 0$ and the differential $\mathsf{d}F_0$ is diagonalizable with no eigenvalues with zero real parts. Denote $E_-$ and $E_+$ respectively the stable and unstable subspaces of $\mathsf{d}F_0$ and $k=\mathrm{dim}(E_-)$
so that $d-k=\mathrm{dim}(E_+)$.
Let us consider the following Cauchy problem
\begin{equation}\label{eq.Cauchy}
H'(t) = F(H(t))\,,\quad t\in\R\,,\quad H(0)=H_0\,,
\end{equation}
and suppose, for any $H_0\in\R^d$, there exists a unique global solution $H\in\mathcal{C}^1(\mathbb{R},\R^d)$.

Then there exists $\mathcal{V}$ a neighborhood of $0$ in $\R^d$ and two $\mathcal{C}^1$-manifolds in $\mathcal{V}$, say $\mathcal{W}_-$ and $\mathcal{W}_+$, with respective dimensions $k$ and $d-k$, tangent at $0$ to $E_-$ and $E_+$ respectively, and such that for all $H_0\in\mathcal{V}$, the solution of \eqref{eq.Cauchy} satisfies:
\begin{enumerate}[\rm (a)]
\item $\lim_{t\to+\infty} H(t) = 0 \Leftrightarrow H_0\in \mathcal{W}_-$,
\item $\lim_{t\to-\infty} H(t) = 0 \Leftrightarrow H_0\in \mathcal{W}_+$.
\end{enumerate}
More precisely, the manifold $\mathcal{W}_-$ (stable at $+\infty$) is the graph of a $\mathcal{C}^1$-map denoted $\varphi_-$ defined over $E_-\cap\mathcal{V}$ and satisfying $\varphi_-(0)=0$ et $\mathsf{d}\varphi_-(0)=0$. The same result holds for $\mathcal{W}_+$.
\end{theorem}
The above result is a classical consequence of a fixed point argument and the convergence to such hyperbolic equilibrium point is exponential in the stable manifold, with a convergence rate similar to the one of the linearized Cauchy problem in the tangent space $E_-$, i.e. of order at worst $\mathcal{O}(e^{-\gamma t})$ where
\[
-\gamma = \max\left(\{\Re(\mathsf{sp}(\mathsf{d}F_0))\}\cap\R_-\right)<0\, .\]

\subsubsection{Convergence to the diagonal matrices: a naive estimate}
From Corollary \ref{cor.cor1'}, we know that $H(t)$ converges to $H_{\infty}=\diag(\alpha_{\ell})$ where the $\alpha_{\ell}$ are \textit{exactly} the eigenvalues of $H_{0}$ since $\Hil$ is of finite dimension. This convergence analysis is then reduced to a small neighborhood of $H_{\infty}$. The differential equation reads
\[H'=F(H)\,,\qquad F(H)=[H, G(H)]\,,\]
and, since $G$ and $H_{\infty}$ are diagonal in the canonical bases, we have $G(H_{\infty})=0$. We deduce that the differential of $F$ at $H_{\infty}$ is given by
\[\forall H\in\mathcal{S}(\Hil)\,,\qquad \mathsf{d}F_{H_{\infty}}(H)=[H_{\infty}, G(H)]\,.\]
Thus $\mathsf{d}F_{H_{\infty}}$ is an endomorphism of the space of symmetric matrices and it is diagonalized in the basis $E_{i,j}$ since a straightforward computation gives
\[\mathsf{d}F_{H_{\infty}}(E_{i,j})=g_{i,j}(\alpha_{i}-\alpha_{j})E_{i,j}\,.\]
Since $0$ is an eigenvalue of $\mathsf{d}F_{H_{\infty}}$ of multiplicity $d$, we cannot directly apply the stable manifold theorem. Nevertheless, we can elementarily get the following proposition.
\begin{proposition}
\label{prop.exp.rate}
Let us introduce
\[\mathcal{T}=\{(i,j)\in\{0,\ldots, d-1\}^2 : i< j\}\,,\]
and assume that 
\[-\gamma:=\displaystyle{\max_{\mathcal{T}}g_{i,j}(\alpha_{i}-\alpha_{j})}<0\,.\]
Then, for all $\varepsilon>0$, there exist $T, C>0$ such that, for all $t\geq T$,
\[\|H(t)-\diag(H(t))\|_{\mathsf{HS}}\leq Ce^{-\gamma t}\,.\]
\end{proposition}
\begin{proof}
Let us define 
\[R(H)=F(H)-F(H_{\infty})-\mathsf{d}F_{H_{\infty}}(H-H_{\infty})=[H-H_{\infty}, G(H-\diag(H))]\,,\]
and
\[\Delta(t)=H(t)-\diag(H(t))\,.\]
Let us now consider $\varepsilon>0$ and $T>0$ such that, for all $t\geq T$, we have
\[\|H(t)-H_{\infty}\|_{\mathsf{HS}}\leq \varepsilon\,,\qquad\|R(H(t))\|_{\mathsf{HS}}\leq 2\varepsilon\|G\|\|\Delta(t)\|_{\mathsf{HS}}\,.\]
For all $H\in\mathcal{S}(\Hil)$, we notice that $\diag\left(\mathsf{d}F_{H_{\infty}}(H)\right)$ is zero. Therefore we have
\[\Delta'(t)=\mathsf{d}F_{H_{\infty}}(\Delta(t))+R(H(t))-\diag\left(R(H(t))\right)\,.\]
The Duhamel formula reads, for $t\geq T$,
\[\Delta(t)=e^{(t-T)\mathsf{d}F_{H_{\infty}}} \Delta(T)+\int_{T}^t e^{(t-s)\mathsf{d}F_{H_{\infty}}}\{R(H(s))-\diag\left(R(H(s))\right)\} \dx s\,,\]
where $\Delta$ and $R(H)-\diag\left(R(H)\right)\in\textsf{Span}\,(E_{i,j})_{i<j}$.
Since $(E_{i,j})$ is orthonormalized for the $\mathsf{HS}$-norm, we deduce that
\[\|\Delta(t)\|_{\mathsf{HS}}\leq e^{-\gamma(t-T)}\|\Delta (T)\|_{\mathsf{HS}}+\int_{T}^t e^{-(t-s)\gamma}\|R(H(s))-\diag\left(R(H(s))\right)\|_{\mathsf{HS}} \dx s\,,\]
so that
\[e^{\gamma t}\|\Delta(t)\|_{\mathsf{HS}}\leq e^{\gamma T}\|\Delta (T)\|_{\mathsf{HS}}+2\varepsilon\|G\|\int_{T}^t e^{s\gamma}\|\Delta(s)\|_{\mathsf{HS}}\dx s\,.\]
From the Gronwall lemma, we get, for $t\geq T$,
\[e^{\gamma t}\|\Delta(t)\|_{\mathsf{HS}}\leq e^{\gamma T}\|\Delta (T)\|_{\mathsf{HS}}e^{2\varepsilon\|G\| (t-T)}\,,\]
and thus
\[\|\Delta(t)\|_{\mathsf{HS}}\leq e^{(-\gamma +2\varepsilon\|G\|)(t-T)}\|\Delta (T)\|_{\mathsf{HS}}\,.\]
We deduce that the same convergence holds for the diagonal terms. 

Let us now eliminate the dependence in $\varepsilon$ in the exponential. We easily have
\[H'(t)=F(H(t))=[H(t),G(H(t))]=[H(t),G(\Delta(t))]\,,\]
so that for all $t\geq T$
\[\|H'(t)\|_{\mathsf{HS}} \leq 2\|G\|\|H_0\|_{\mathsf{HS}}\|\Delta(t)\|_{\mathsf{HS}} \leq 2\|G\|\|H_0\|_{\mathsf{HS}} C e^{(-\gamma+\epsilon)t}\,.\]
It follows that
\[\|H(t)-H_{\infty}\|_{\mathsf{HS}}\leq \int_t^{+\infty} \|H'(s)\|_{\mathsf{HS}}\dx s\leq 2\|G\|\|H_0\|_{\mathsf{HS}} \dfrac{C}{-\gamma+\epsilon} e^{(-\gamma+\epsilon)t}\,.\]
We get therefore the existence of a constant $D$ that depends on $\epsilon$ and $\|H_0\|_{\mathsf{HS}}$ only such that for all $t\geq T$,
\[\|H(t)-H_{\infty}\|_{\mathsf{HS}}\leq D e^{(-\gamma+\epsilon)t}\,.\]
Coming back to the previous proof with $De^{(-\gamma+\epsilon)t}$ in place of $\epsilon$, we get then from the Gronwall lemma
\[e^{\gamma t}\|\Delta(t)\|_{\mathsf{HS}}\leq e^{\gamma T}\|\Delta (T)\|_{\mathsf{HS}}e^{2D\|G\|\int_T^te^{(-\gamma+\epsilon)s}\dx s}\,,\]
that may be bounded independently from $t\geq T$. Therefore we may omit the $\epsilon$ term in exponential convergence rate.

\end{proof}

\begin{remark}
In particular, this proposition may be used as follows. Let us consider $\varepsilon$ such that $\displaystyle{0<\varepsilon<\inf_{\mathcal{T}}|\alpha_{i}-\alpha_{j}|}$ and assume that $\|H_{0}-H_{\infty}\|_{\mathsf{HS}}<\varepsilon$ which may be insured by using the flow up to the time $T$. Then, there exists a permutation $\Pi$ such that
\[\|\tilde H_{0}-\tilde H_{\infty}\|_{\mathsf{HS}}=\mathcal{O}(\varepsilon)\,,\quad \tilde H_{0}=\Pi^\star H_{0}\Pi\,,\quad \tilde H_{\infty}=\diag(\tilde \alpha_{\ell})\,,\]
where $(\tilde \alpha_{\ell})_{\ell\in\{0,\ldots, d\}}$ denotes the sequence of the eigenvalues of $H_{0}$ in the opposite order with respect to $g_{i,j}$. Then we have
\[\|\tilde H(t)-\diag(\tilde H(t))\|_{\mathsf{HS}}\leq Ce^{-\gamma t}\,.\]
In other words, up to a permutation and as soon as we are close enough to the limit, we can insure an exponential convergence (with rate $\gamma$) to the diagonal matrices.
\end{remark}

\subsubsection{Proof of Corollary \ref{cor.dim.finie}}\label{sec.conv.exp}
It is actually more convenient to work on the side of the unitary transform $U$. We recall that, from \eqref{eq.U}, $U$ satisfies 
\[U'=UG(U^\star H_{0}U)=\mathcal{F}(U)\,.\]
Let us first notice that $U$ converges.
\begin{proposition}\label{eq.CVU}
The function $t\mapsto U(t)$ converges as $t$ tends to $+\infty$. Moreover, the limit $U_{\infty}$ belongs to the finite set
\[\left\{\Pi^\star\diag(\epsilon_{\ell}) \Pi\,,\quad \Pi\in\mathfrak{S}_{n}\,, \epsilon_{\ell}\in\{\pm 1\}\right\}\,.\]
\end{proposition}
\begin{proof}
The orthogonal group is compact since the dimension is finite. If $U$ converges to $U_{\infty}$, then $U_{\infty}$ is still unitary and $H_{\infty}=U_{\infty}^\star H_{0} U_{\infty}$. There exists an orthogonal transformation $Q_{0}$ such that $Q_{0} H_{0}Q^\star_{0}=\diag(\alpha_{\ell})$ where the $\alpha_{\ell}$ are in non-increasing order. Thus, there exists a permutation $\Pi$ such that
\[H_{\infty}=(U_{\infty}^\star Q^\star_{0}\Pi^\star) H_{\infty} (\Pi Q_{0} U_{\infty})\,,\]
Thus, since $H_{\infty}$ is diagonal with simple coefficients, $\Pi Q_{0} U_{\infty}$ is itself diagonal and thus, we get $(\Pi Q_{0} U_{\infty})^2=\mathrm{Id}$.

Now we know that $H(t)$ tends to $H_{\infty}$ so that $U'(t)$ goes to $0$ as $t$ goes to $+\infty$. Since the orthogonal group is compact and that $\mathcal{F}$ is continuous with a finite number of zeros, we deduce that $U$ must converge to a zero of $\mathcal{F}$.
\end{proof}
We can notice that the proof of Proposition \ref{eq.CVU} relies on the convergence of $H(t)$ when $t$ goes to $+\infty$. Without this a priori convergence, we would not get the convergence of $U(t)$ for any given $H_{0}$. As in \cite[Theorem 3]{Bro91}, one should exclude from the initial conditions the union of the stable manifolds of all the stationary points that are not attractive. With the compactness of $\mathcal{U}(\Hil)$, this would imply the convergence of $U(t)$ towards the (unique) attractive point. Actually such restrictions on the initial conditions are not necessary to get the convergence as explained in the following lines.

Let us now parametrize, locally near $U_{\infty}$, the orthogonal group $\mathcal{U}(\Hil)$ by its Lie algebra $\mathcal{A}(\Hil)$. By the local inversion theorem, it is well known that $\exp : \mathcal{A}(\Hil)\to \mathcal{U}(\Hil)$ is a local smooth diffeomorphism between a neighborhood of $0$ and a neighborhood of $\mathrm{Id}$. We recall that
\[(\mathsf{d}\exp (A))^{-1}(H)=\exp(-A)\sum_{k=0}^{+\infty} \frac{(\mathsf{ad}A)^k}{(k+1)!} H\,.\]
For $t\geq T$, we may write $U(t)=U_{\infty}\exp(A(t))$. Thus we have
\[U'=U_{\infty}e^{A(t)}G(e^{-A(t)}H_{\infty}e^{A(t)})\,.\]
We get
\[A'(t)=e^{A(t)} G(e^{-A(t)}H_{\infty}e^{A(t)})e^{-A(t)}+R_{1}(t)\,,\]
with $R_{1}(t)=\mathcal{O}(\|A(t)\|^2)$. We infer that
\[A'(t)= G([H_{\infty},A])+R_{2}(t)\,,\]
with $R_{2}(t)=\mathcal{O}(\|A(t)\|^2)$. Now $A\mapsto G([H_{\infty},A])$ is an endomorphism of the space of skew-symmetric matrices $\mathcal{A}(\Hil)$ and it is diagonalized in the basis $E^\pm_{i,j}$ with $i<j$ with eigenvalues $g_{i,j}(\alpha_{i}-\alpha_{j})$. By assumption these eigenvalues are not zero. We may apply the stable manifold theorem. Since $A(t)$ goes to $0$, this means, by definition, that $A(t)$ belongs to the stable manifold near $0$. Therefore it goes exponentially to zero, with a convergence rate a priori of order $\mathcal{O}(e^{-\gamma t})$, with $\displaystyle{\gamma=\inf_{\mathcal{T}^-}|g_{i,j}(\alpha_{i}-\alpha_{j})|>0}$. Of course, the convergence may be stronger, depending on the stable manifold on which $A(t)$ is evolving. Then $H(t)$ inherits this convergence rate.

\subsection{Abstract factorization of Hilbert-Schmidt operators}
Actually it is possible to get a quantitative convergence to the eigenvalues of non self-adjoint operators in infinite dimension. Note that, in Theorem \ref{gen-theo}, we do not have such a quantification, even for symmetric operators.
\subsubsection{Chu-Norris decomposition for Hilbert-Schmidt operators}
Let us consider an application $G : \mathcal{L}(\Hil)\to\mathcal{A}(\Hil)$ of class $\mathcal{C}^1$ and the following Cauchy problem:
\[H'(t)=[H(t), G(H(t))]\,,\qquad H(0)=H_{0}\in\mathcal{L}_{2}(\Hil)\,.\]
We also introduce the equations
\begin{align*}
g_{1}'(t)&=g_{1}(t)G(H(t))\,, &g_{1}(0)=\mathrm{Id}\,,\\
g'_{2}(t)&=(H(t)-G(H(t)))g_{2}(t)\,,&g_{2}(0)=\mathrm{Id}\,.
\end{align*}
Then we have the following lemma (coming from a straightforward adaptation of the theorems in \cite{CN88} that are actually true for bounded operators).
\begin{lemma}\label{lemma.ChuNorris}
The solutions $H, g_{1}$ and $g_{2}$ are global and of class $\mathcal{C}^1$. Then, we have $g_{1}\in\mathcal{C}^1(\R,\mathcal{U}(\Hil))$ and 
\begin{equation}\label{eq.g1g2}
H(t)=g_{1}(t)^{-1}H_{0}g_{1}(t)=g_{2}(t)H_{0}g_{2}^{-1}(t)
\end{equation}
and thus we have $H\in\mathcal{C}^1(\R,\mathcal{L}_{2}(\Hil))$. 

Moreover we have the relations
\begin{equation}\label{eq.g1g2'}
e^{tH_{0}}=g_{1}(t)g_{2}(t)\,,\qquad e^{tH(t)}=g_{2}(t)g_{1}(t)\,.
\end{equation}
\end{lemma}
\begin{proof}
Let us just give some insights of the proof. By easy calculations, we get $(g_{1}g_{1}^\star)'=0$, $(g_{1}Hg_{1}^\star)'=0$, and $(g_{2}^{-1}Hg_{2})'=0$, which ensures both the first algebraic properties and the global character of the solutions in the considered spaces. The first exponential factorization follows from the uniqueness result for the following Cauchy problem
\[Z'(t) = H_0Z(t),\quad t\in\R,\quad Z(0)=\mathrm{Id}\,,\]
of which both $e^{tH_0}$ and $g_1(t)g_2(t)$ are solutions. The second exponential factorization then easily follows as a consequence.
\end{proof}

\subsubsection{Proof of Proposition \ref{prop.House}}
It is easy to see that 
\[g_{2}(t) e_{\ell}=\sum_{j=0}^{\ell} \alpha_{\ell, j}(t)e_{j}\,.\]
Thus we may write
\[e^{t H_{0}}e_{\ell}=\sum_{j=0}^\ell \alpha_{\ell, j}(t) g_{1}(t)e_{j}\,.\]
We get
\[P^{-1}e^{t \Lambda}f_{\ell}=\sum_{j=0}^\ell \alpha_{\ell, j}(t) g_{1}(t)e_{j}\,,\qquad f_{\ell}=P e_{\ell}\,.\]
We have 
\begin{equation}\label{eq.alpha}
\sum_{j=0}^\ell \alpha_{\ell, j}(t) g_{1}(t)e_{j}=P^{-1}e^{t \Lambda}f_{\ell}=\sum_{j=0}^J p_{\ell, j} e^{t\lambda_{j}}v_{j}+\sum_{j=J+1}^{+\infty} p_{\ell, j}e^{t\lambda_{j}} v_{j}\,,
\end{equation}
where $v_{j}=P^{-1}e_{j}$. We may also write
\begin{equation}\label{eq.alpha'}
A_{J}\left[\begin{array}{c}e_{0}\\ \vdots\\ e_{J} \end{array}\right] = P_{J}e^{t\Lambda_{J}}\left[\begin{array}{c}g^\star_{1}v_{0}\\ \vdots\\ g^\star_{1}v_{J} \end{array}\right]+\mathcal{O}(e^{t\Re(\lambda_{J+1})})\,.
\end{equation}
and
\[ A_{J} = P_{J}e^{t\Lambda_{J}}\left[\begin{array}{c}g^\star_{1}v_{0}\\ \vdots\\ g^\star_{1}v_{J} \end{array}\right][e_{0}\ldots e_{J}]+\mathcal{O}(e^{t\Re(\lambda_{J+1})})\,.\]
By using the simplicity of the eigenvalues, we get, by using a Neumann series and the fact that $P_{J}$ is invertible, that $A_{J}$ is invertible when $t$ is large enough. In addition, we get
\begin{equation}\label{eq.AJ-1}
\|A_{J}^{-1}\|=\mathcal{O}\left(e^{-t\Re(\lambda_{J})}\right)\,.
\end{equation}
From \eqref{eq.alpha}, we get
\begin{equation*}
\sum_{j=0}^\ell \alpha_{\ell, j}(t) H_{0}g_{1}(t)e_{j}=\sum_{j=0}^J p_{\ell, j} \lambda_{j}e^{t\lambda_{j}}v_{j}+\mathcal{O}\left(e^{t\Re(\lambda_{J+1})}\right)\,,
\end{equation*}
so that,  with \eqref{eq.g1g2},
\begin{equation*}
\sum_{j=0}^\ell \alpha_{\ell, j}(t) g_{1}(t)H(t)e_{j}=\sum_{j=0}^J p_{\ell, j} \lambda_{j}e^{t\lambda_{j}}v_{j}+\mathcal{O}\left(e^{t\Re(\lambda_{J+1})}\right)\,,
\end{equation*}
which may be written as
\[ A_{J}\left[\begin{array}{c}H(t)e_{0}\\ \vdots\\ H(t)e_{J} \end{array}\right] = P_{J}\Lambda_{J}e^{t\Lambda_{J}}\left[\begin{array}{c}g^\star_{1}v_{0}\\ \vdots\\ g^\star_{1}v_{J} \end{array}\right]+\mathcal{O}(e^{t\Re(\lambda_{J+1})})\,.\]
or, with \eqref{eq.AJ-1},
\[ \left[\begin{array}{c}H(t)e_{0}\\ \vdots\\ H(t)e_{J} \end{array}\right] = A^{-1}_{J}P_{J}\Lambda_{J}e^{t\Lambda_{J}}\left[\begin{array}{c}g^\star_{1}v_{0}\\ \vdots\\ g^\star_{1}v_{J} \end{array}\right]+\mathcal{O}(e^{t\Re(\lambda_{J+1}-\lambda_{J})})\,.\]
From \eqref{eq.alpha'}, we get
\begin{equation}\label{eq.alpha''}
\left[\begin{array}{c}g^\star_{1}v_{0}\\ \vdots\\ g^\star_{1}v_{J} \end{array}\right]=e^{-t\Lambda_{J}}P^{-1}_{J}A_{J}\left[\begin{array}{c}e_{0}\\ \vdots\\ e_{J} \end{array}\right]+\mathcal{O}(e^{t\Re(\lambda_{J+1}-\lambda_{J})})\,,
\end{equation}
and thus
\begin{equation}\label{eq.approx-inv}
 \left[\begin{array}{c}H(t)e_{0}\\ \vdots\\ H(t)e_{J} \end{array}\right] = A^{-1}_{J}P_{J}\Lambda_{J}P^{-1}_{J}A_{J}\left[\begin{array}{c}e_{0}\\ \vdots\\ e_{J} \end{array}\right]+\mathcal{O}(e^{t\Re(\lambda_{J+1}-\lambda_{J})})\,.
\end{equation}
Therefore the space spanned by $e_{0},\ldots, e_{J}$ is invariant by $H(t)$ up to an error of  order $\mathcal{O}(e^{t\Re(\lambda_{J+1}-\lambda_{J})})$. In view of \eqref{eq.approx-inv}, we get the conclusion by induction on $J\geq 0$. In other words, we approximate the diagonal terms of $H(t)$ one by one by moving $J$ from $0$ to $+\infty$.

\subsection{Numerics}

In the following numerical illustrations, we consider the finite dimensional example of flows evolving in matrices of $\mathcal{M}_{d}(\R)$ with $d=5$. As an initial data, let fix $H_0=Q\,\diag(\ell^2)_{1\leq \ell\leq d}\,Q^\star$, where $Q=\exp(B_5)$ is a unitary matrix defined through the following choice of $B_d\in\mathcal{A}_{d}(\R)$ skew-symmetric:
\[B_d = (b_{i,j}),\quad b_{i,j}=1,\textrm{ for }i>j\,.\]
The numerical solver to compute any of the below approximate solutions is an adaptive 4th-order Runge-Kutta scheme.

\subsubsection{Brockett's choice} 
The considered Brockett's bracket flow is given by the choice $A = \diag(d-\ell)_{0\leq \ell\leq d-1}\in\mathcal{D}(\R^d)$ with non-increasing diagonal elements. With these data, we observe the convergence to the limit $H_{\infty} = \diag([25,16,9,4,1])$, see Figure~\ref{figure.brockett}. The diagonal terms are sorted in a descending order, according to the coefficients in $A$. The effective exponential rate of convergence $\gamma=3$ coincides, for sufficiently large times, with the worst expected one (Corollary~\ref{cor.dim.finie}).
The matrix of $g_{i,j}(\alpha_i-\alpha_j)$ is there:
\[\begin{pmatrix}
0  &  - 9  & - 32&  - 63 & - 96  \\
  - 9   &  0  &  - 7 &  - 24 & - 45  \\
  - 32 & - 7 &   0 &  - 5 &  - 16  \\
  - 63 & - 24 & - 5  &   0  &  - 3   \\
  - 96 & - 45 & - 16 & - 3   &  0
  \end{pmatrix}\,.\]
\begin{figure}
\includegraphics[scale=0.32]{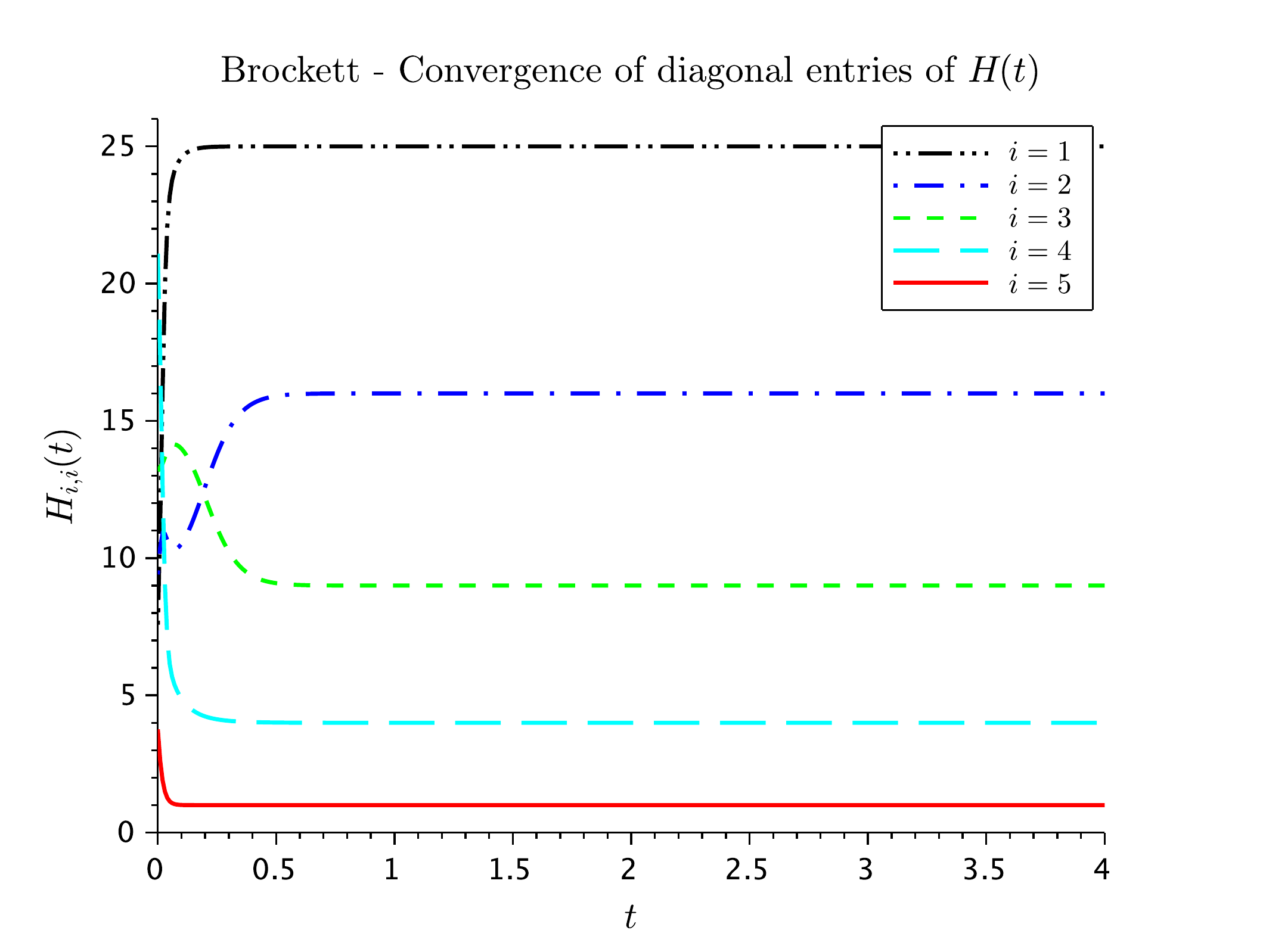}
\includegraphics[scale=0.32]{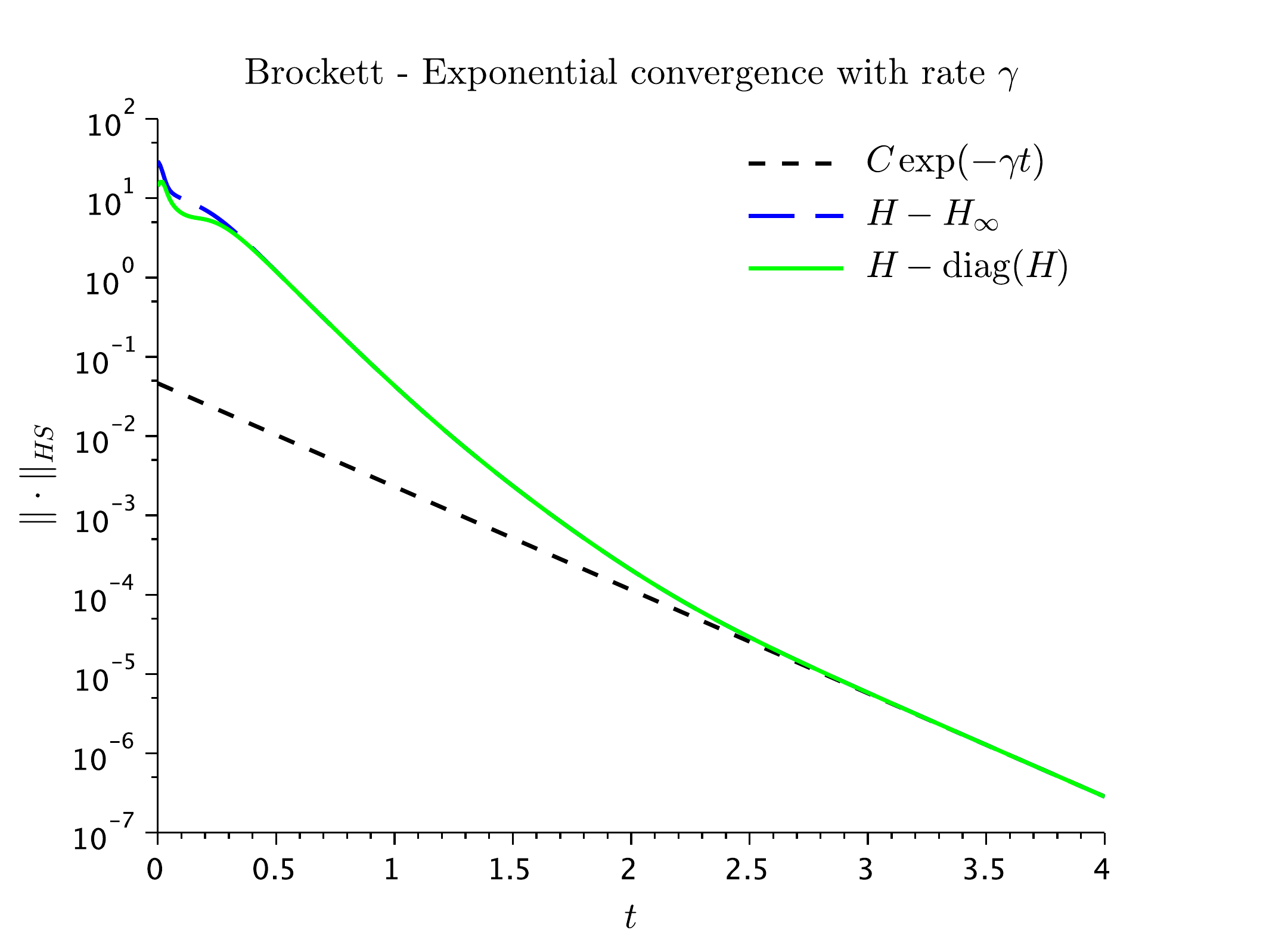}
\caption{Brockett flow -- Convergence of diagonal entries (left), exponential convergence of $H(t)$ (right).}
\label{figure.brockett}
\end{figure}

Let change the initial data to the following one $\tilde{H}_0=\tilde{Q}\,\diag(\ell^2)_{1\leq \ell\leq d}\,\tilde{Q}^\star$, where $\tilde{Q}=\begin{pmatrix}1 & (0)\\ (0) & \exp(B_4)\end{pmatrix}$. Then $\tilde{H}_0$ takes the above block form
\begin{equation}\label{form.sev}
\begin{pmatrix}1 & (0) \\ (0) & K\end{pmatrix}\,,
\end{equation}
where $K\in\mathcal{S}_{4}(\R)$ has spectrum $\{4,9,16,25\}$. Then we observe the convergence to the limit $H_{\infty} = \diag([1,25,16,9,4])$, with an exponential rate $\gamma=5$. Namely, the matrix of $g_{i,j}(\alpha_i-\alpha_j)$ is then:
\[\begin{pmatrix}
   0   &   24   & 30  &  24  &  12  \\
    24  &  0  &  - 9  & - 32&  - 63  \\
    30 & - 9   &  0  &  - 7 &  - 24  \\
    24 & - 32 & - 7&     0  &  - 5  \\
    12 & - 63 & - 24 & - 5  &   0
\end{pmatrix}\,.\]
In fact, in that case the exact flow evolves in the linear subspace described in~\eqref{form.sev} and the positive coefficients of the above matrix concern precisely the evolution of the flow transversally to that subspace. It is only because of the very specific form of the initial data that the numerical flow evolves in the same way, elsewhere numerical error due to the finite precision would have induced a convergence of the numerical solution to the totally sorted limit with rate $\gamma=3$.

\subsubsection{Toda's choice} 
For the same test case, the Toda flow gives similar results, see Figure~\ref{figure.toda}. Once again, the effective exponential rate of convergence is $\gamma=3$ and coincides, for sufficiently large times with the theoretical one.
Actually the matrix of $g_{i,j}(\alpha_i-\alpha_j)$ is there:
\[\begin{pmatrix}
    0   &  -9 &    -16  &  -21  &  -24  \\
  - 9    & 0 &    -7   &  -12  &  -15  \\
  - 16&  - 7   &  0  &   -5  &   -8   \\
  - 21&  - 12&  - 5   &  0  &   -3   \\
  - 24&  - 15&  - 8&   - 3 &    0
    \end{pmatrix}\,.\]
On Figure~\ref{figure.toda2}, we present the numerical counterpart of Proposition~\ref{prop.House}. For any $1\leq \ell\leq d-1$, we compute the norm of the residual column $\|(h_{j,\ell}-\alpha_\ell\delta_{j,\ell})_{\ell\leq j\leq d}\|$ along the time. The thick dashed curves corresponds to the numerical computations and the thin continuous ones corresponds to the expected rate of convergence, namely $\Re(\alpha_{\ell}-\alpha_{\ell+1})\in\{9,7,5,3\}$ successively, because of the precise distribution of the eigenvalues of $H_0$ in this example.
\begin{figure}
\includegraphics[scale=0.32]{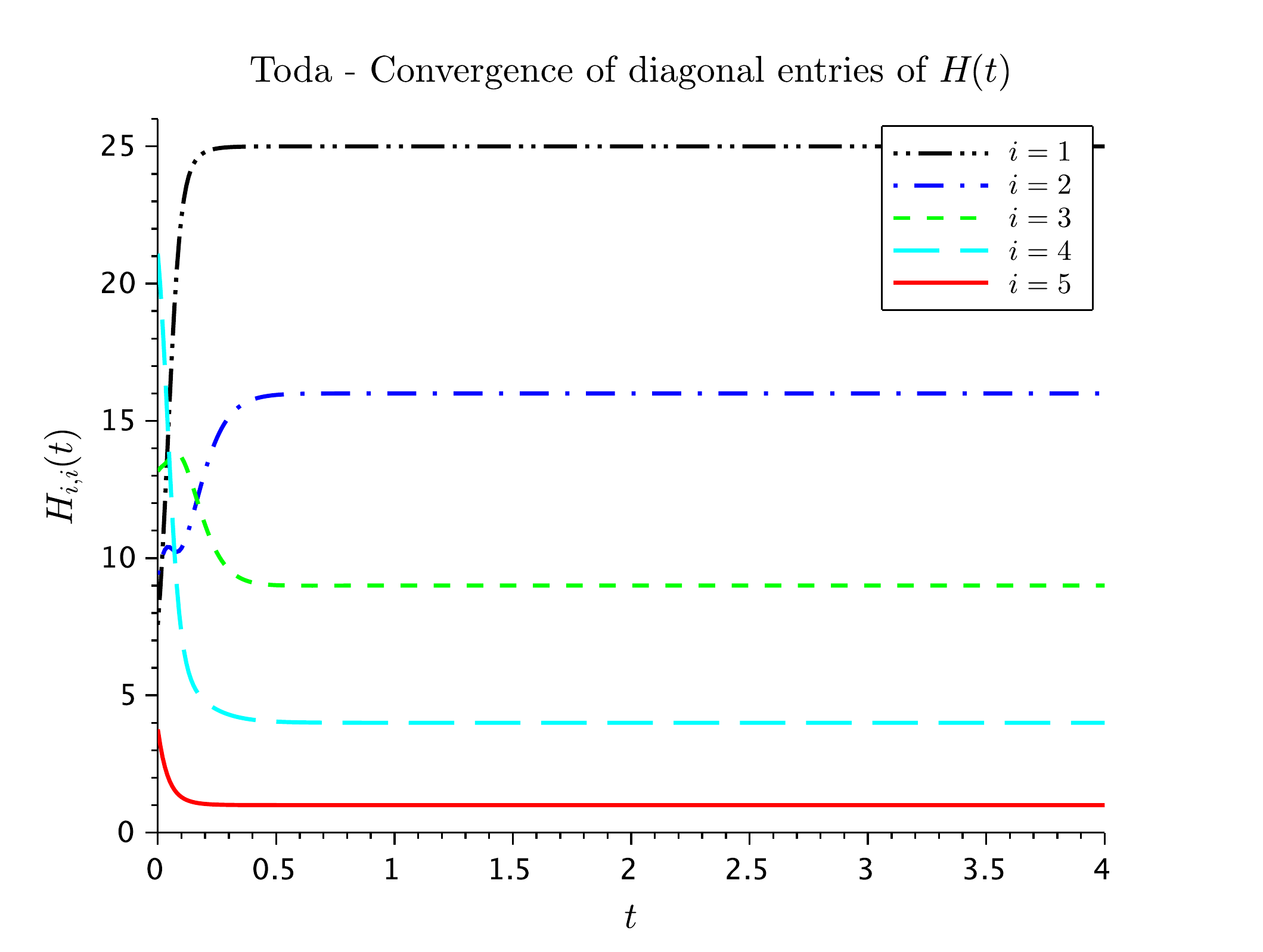}
\includegraphics[scale=0.32]{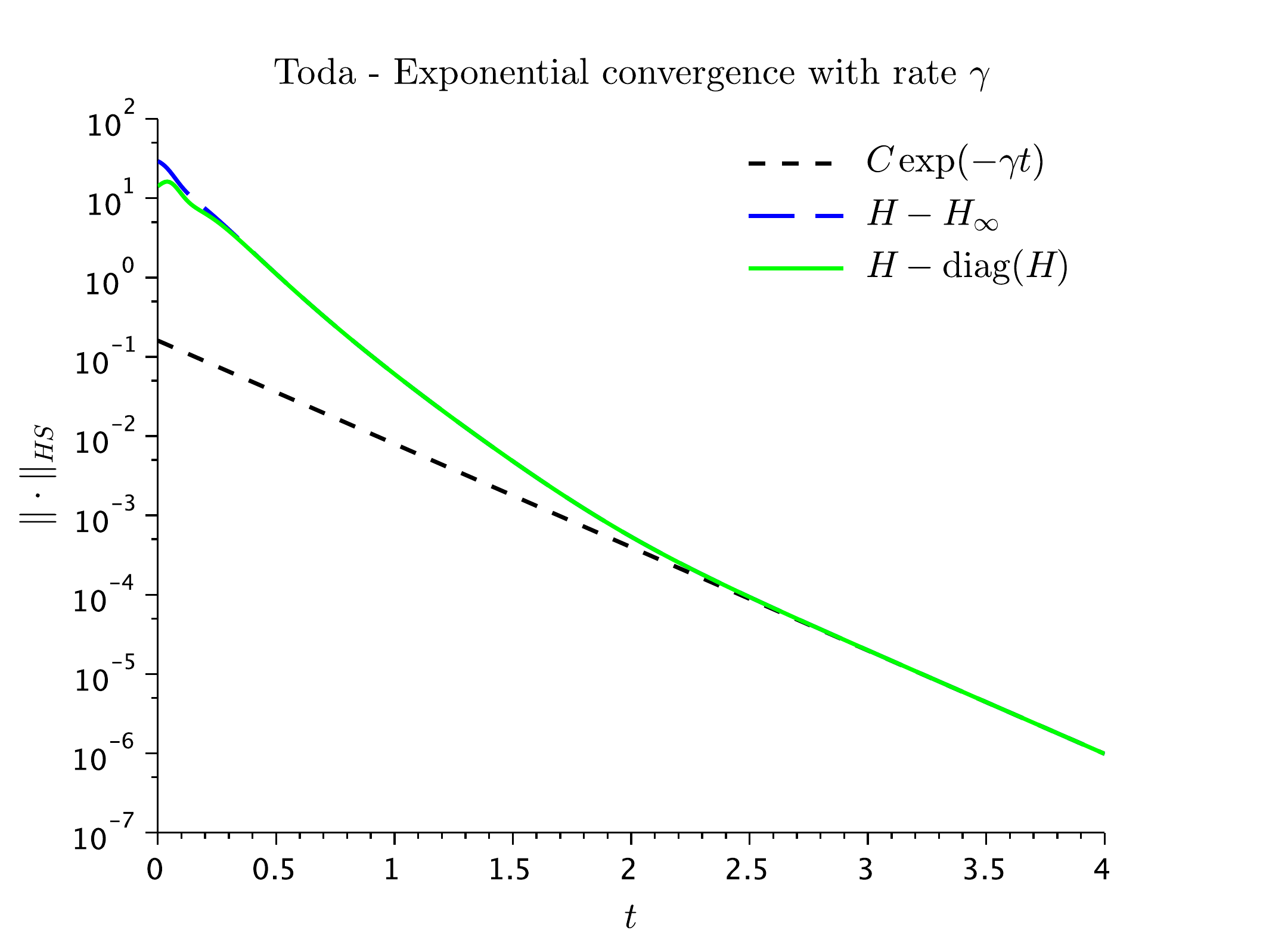}
\caption{Toda flow -- Convergence of diagonal entries (left), exponential convergence of $H(t)$ (right).}
\label{figure.toda}
\end{figure}
\begin{figure}
\includegraphics[scale=0.32]{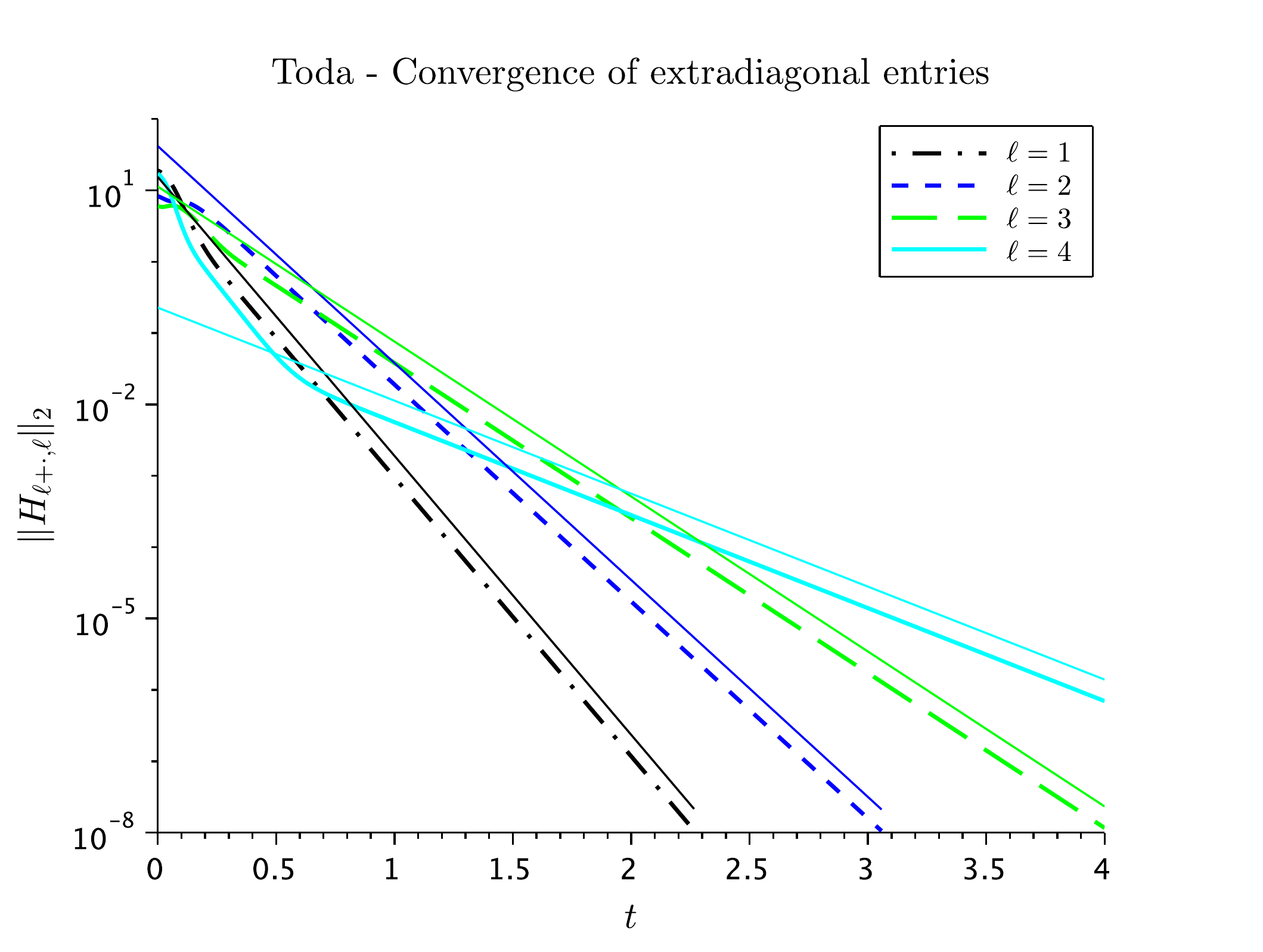}
\caption{Toda flow -- Convergence of extradiagonal entries.}
\label{figure.toda2}
\end{figure}

\subsubsection{Relation with the QR algorithm}
The exponential relations in Lemma~\ref{lemma.ChuNorris} are the keystone for the connection of the Toda flow to the well-known QR algorithm. Under assumptions of Proposition~\ref{prop.House}, $\exp(H_0)$ is invertible and diagonalizable with eigenvalues $(e^{\lambda_j})_{j\geq 1}$ with decreasing moduli. This is a sufficient condition for the convergence of the QR algorithm applied to $e^{H_0}$. We already get $g_1(t)\in\mathcal{U}(\Hil)$ and, for that flow, the Cauchy-Lipschitz theorem ensures that $g_2$ takes value in the subset of upper triangular matrices. At $t=1$, $e^{H_0}=g_1(1)g_2(1)$ is therefore nothing but a QR factorization of $e^{H_0}$. Thus $e^{H(1)}=g_2(1)g_1(1)$ corresponds to the first iteration of the discrete algorithm and, since the differential equation is autonomous, that handling iterates along integer times so that the sampling sequence $(e^{H(n)})$ mimics the QR algorithm applied to $e^{H_0}$. In fact, the diagonal coefficients of $g_2$ are not all positive (even if they would be, up to a product with a $\diag(\pm 1)$ matrix, that corresponds to a change in the initial data for $g_1$ and $g_2$ depending on $H_0$), and the two algorithms slightly differ.

\subsubsection{Wegner's choice} 
Consider now the above initial data for the Cauchy problem associated to Wegner's flow. The solution converges to $H_{\infty}=\diag([4,9,16,25,1])$. The exponential rate to the limit is $\gamma=-9$ that matches the expected value. Indeed a simple calculation gives, for $E\in\mathcal{S}(\Hil)$
\[\textsf{d}F_{H_{\infty}}(E) =  [H_\infty,[E,H_{\infty}]]\,,\]
so that for $i<j$
\[\textsf{d}F_{H_{\infty}}(E_{i,j}) = -(\alpha_i-\alpha_j)^2E_{i,j}\,.\]
Therefore, at the limit $H_\infty$, one has the following "matrix-eigenvalue":
\[\begin{pmatrix}
    0  & - 25  & - 144  &- 441 & - 9  \\ 
  - 25  &   0  & - 49 &  - 256  &- 64   \\
  - 144 & - 49  &   0  & - 81   &- 225  \\
  - 441  &- 256 & - 81  &   0  & - 576  \\
  - 9    &- 64   &- 225 & - 576  &  0
\end{pmatrix}\,.\]
\begin{figure}
\includegraphics[scale=0.32]{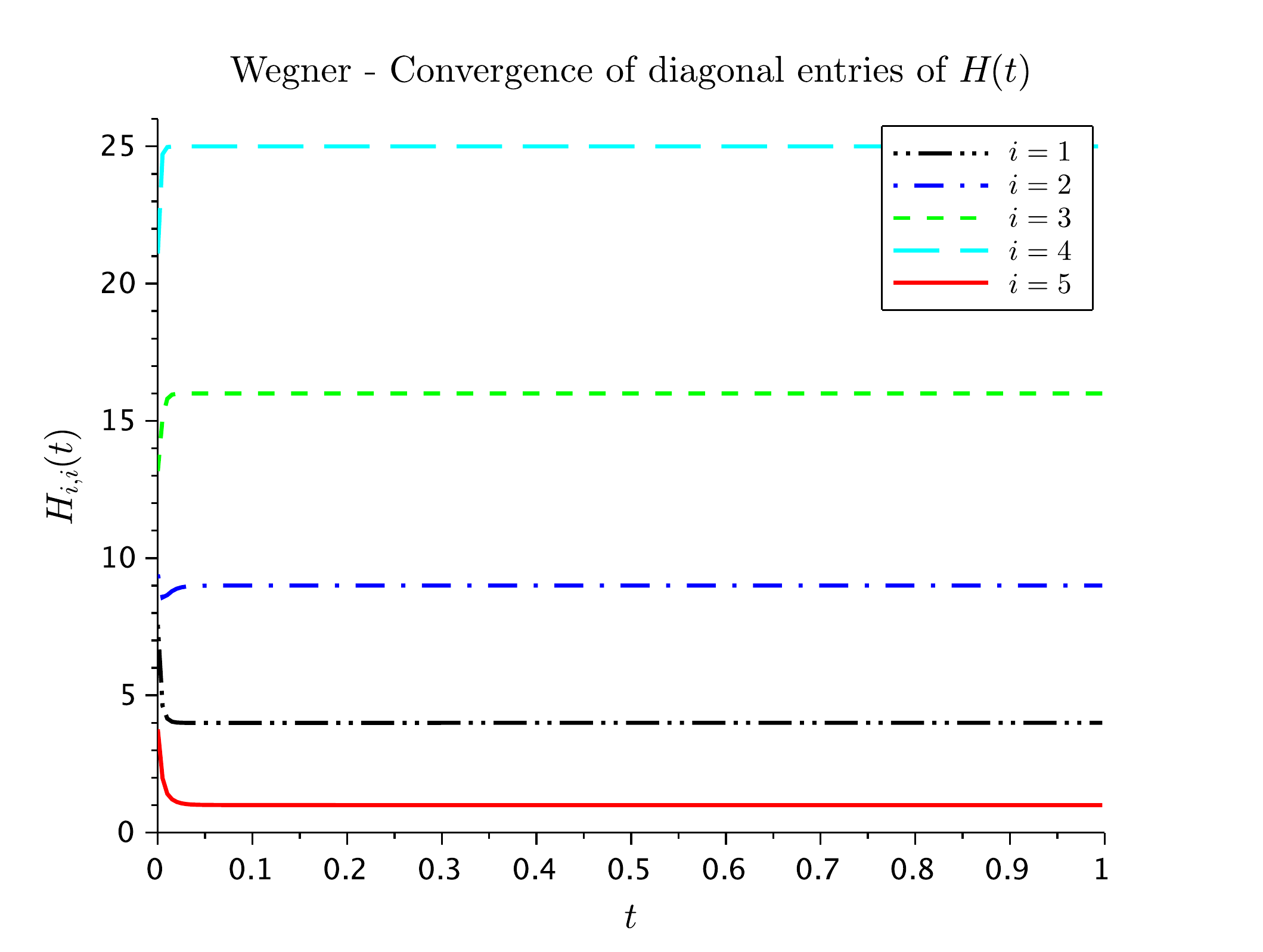}
\includegraphics[scale=0.32]{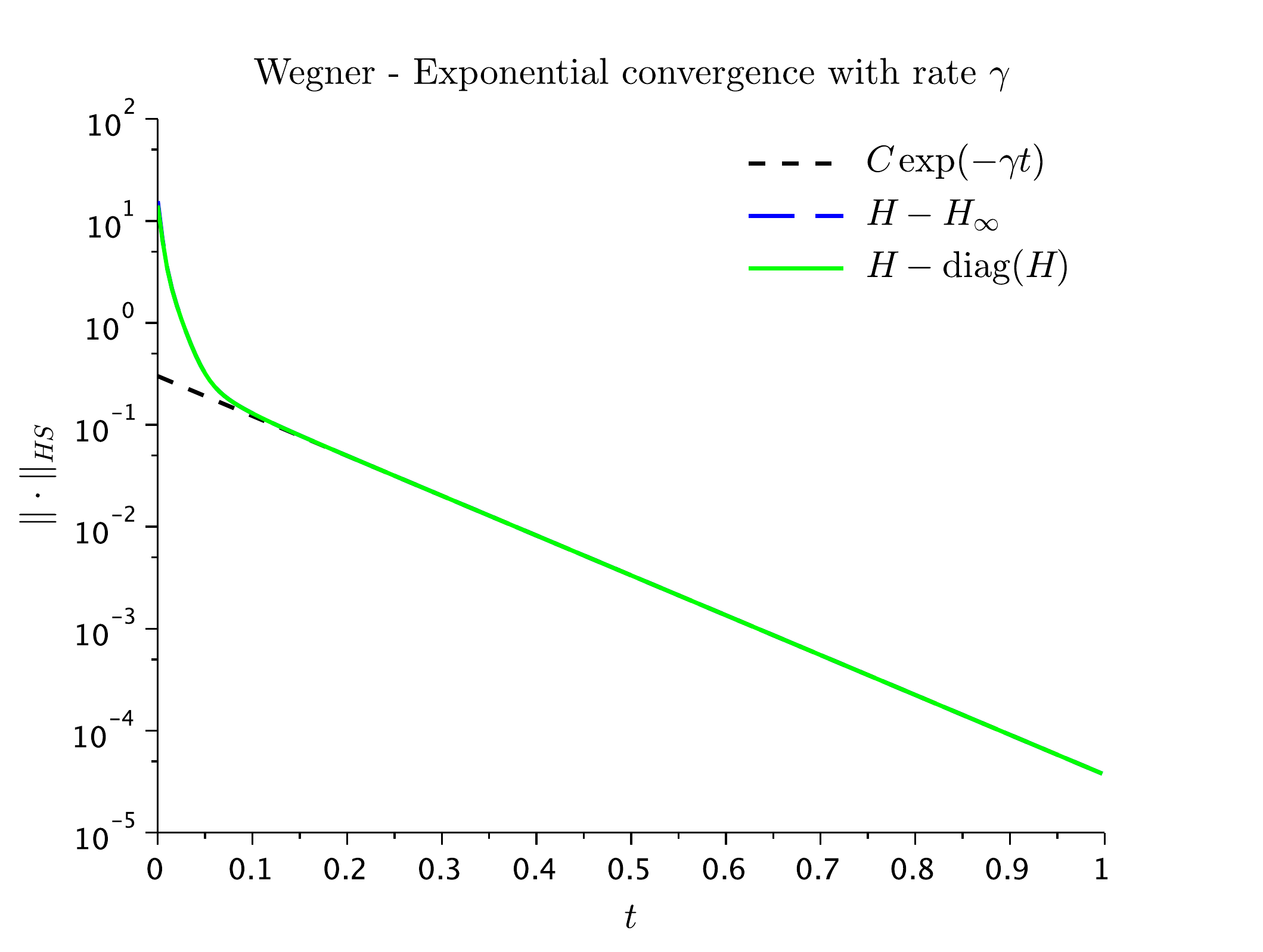}
\caption{Wegner flow -- Convergence of diagonal entries (left), exponential convergence of $H(t)$ (right).}
\label{figure.wegner}
\end{figure}
In that case, for any possible diagonal limit $H_\infty$, the corresponding eigenvalues of $\textsf{d}F_{H_{\infty}}$ are all negative, except 0 that is an eigenvalue of multiplicity $d$.

\end{document}